\documentclass[12pt,a4paper,reqno]{amsart}
\usepackage{lineno}
\allowdisplaybreaks
\usepackage{amsmath}
\usepackage{amssymb}
\usepackage{amsfonts}
\usepackage{amsthm,amsfonts,amsthm,latexsym,enumerate,url,cases}
\numberwithin{equation}{section}
\usepackage{mathrsfs}
\usepackage{hyperref}
\hypersetup{colorlinks=true,citecolor=blue,linkcolor=blue,urlcolor=blue}
\def\pmod #1{\ ({\rm{mod}}\ #1)}

\theoremstyle{plain}
\newtheorem{theorem}{Theorem}

\newtheorem{lemma}{Lemma}

\newtheorem{corollary}{Corollary}
\newtheorem{proposition}{Proposition}
\newtheorem{conjecture}{Conjecture}
\theoremstyle{definition}

\usepackage{etoolbox}
\makeatletter
\patchcmd{\@settitle}{\uppercasenonmath\@title}{}{}{}
\patchcmd{\@setauthors}{\MakeUppercase}{}{}{}
\patchcmd{\section}{\scshape}{}{}{}
\makeatother

\begin{document}

\title
[Note on a conjecture of S\'ark\"ozy on special sequences]
{Note on a conjecture of S\'ark\"ozy on special sequences}

\author
[Y. Ding, H. Li and Z. Zhang] {Yuchen Ding, Huixi Li and Zihan Zhang}

\address{(Yuchen Ding$^{1,2}$) $^1$School of Mathematics,  Yangzhou University, Yangzhou 225002, People's Republic of China}
\address{$^2$HUN-REN Alfr\'ed R\'enyi Institute of Mathematics, Budapest, Pf. 127, H-1364 Hungary}
\email{ycding@yzu.edu.cn}

\address{(Huixi Li) School of Mathematical Sciences and LPMC, Nankai University, Tianjin 300071, People's Republic of China}
\email{lihuixi@nankai.edu.cn}

\address{(Zihan Zhang) School of Mathematical Sciences and LPMC, Nankai University, Tianjin
300071, People’s Republic of China}
\email{2211056@mail.nankai.edu.cn}


\begin{abstract}
Let $\alpha>1$ be an irrational number and $k\ge 2$ a positive integer. Let $f(x)$ be a polynomial with positive integer coefficients. Solving a 2001 problem of S\'ark\"ozy on special sequences, Hegyv\'ari proved in 2003 that there exists an infinite sequence $A$ with density $\frac{1}{k}-\frac{1}{k\alpha}$ such that
$$
\big\{f(a_1)+\ldots+f(a_k): a_i\in A, 1\le i\le k\big\}\cap \big\{\lfloor n\alpha\rfloor: n\in \mathbb{N}\big\}=\emptyset.
$$
Hegyv\'ari  also proved that the density given by him is optimal for $k=2$. In this article, we show that the density $\frac{1}{k}-\frac{1}{k\alpha}$ given by Hegyv\'ari is actually optimal for all $k\ge 2$. 
\end{abstract}
\maketitle

In section 6 of his 2001 problem list, S\'ark\"ozy \cite[Conjecture 47]{Sarkozy} posed the following attractive conjecture involving special sequences.

\begin{conjecture}\label{conjecture1}
If $A$ and $B$ are infinite sequences of positive integers, the lower asymptotic densities of $A$ and $B$ are positive, and $\alpha$ is an irrational number, then 

$1)$ $ab=\lfloor n\alpha\rfloor, \quad a\in A,~ b\in B,$

$2)$ $a^2+b^2=\lfloor n\alpha\rfloor, \quad a\in A,~ b\in B$\\
have infinitely many solutions, respectively.
\end{conjecture}

Let $\mathbb{N}$ be the set of natural numbers. We need only to consider the case $\alpha>1$ in Conjecture \ref{conjecture1}, since $\big\{\lfloor n\alpha\rfloor: n\in \mathbb{N}\big\}=\mathbb{N}$ for $0<\alpha< 1$, in which case Conjecture \ref{conjecture1} is correct trivially.

Hegyv\'ari \cite{Hegyvari} answered Conjecture \ref{conjecture1} negatively. Surprisingly, as illustrated by Hegyv\'ari even a weaker form of Conjecture \ref{conjecture1} fails.
Let $\alpha>1$ be an irrational number and $k\ge 2$ a positive integer. Let $f(x)$ be a polynomial with positive integer coefficients. Then Hegyv\'ari proved that there exists an infinite sequence $A$, for which the density of $A$ exists, $d(A)=\frac{1}{k}-\frac{1}{k\alpha}>0$ and
\begin{align}\label{eq-1-1}
\big\{f(a_1)+\ldots+f(a_k): a_i\in A, 1\le i\le k\big\}\cap \big\{\lfloor n\alpha\rfloor: n\in \mathbb{N}\big\}=\emptyset,
\end{align}
where $d(A)$ is the density of $A$. Hegyv\'ari \cite{Hegyvari} also showed that the density $\frac{1}{k}-\frac{1}{k\alpha}$ he provided is optimal in the case $k=2$. Specifically,  let $A$ be an infinite sequence of positive integers with $\overline{d}(A)>\frac{1}{2}-\frac{1}{2\alpha}$, where $\overline{d}(A)$ is the upper asymptotic density of the set $A$. Then Hegyv\'ari proved that 
\begin{align}\label{eq-1-2}
\Big|\big\{f(a_1)+f(a_2): a_i\in A, i=1, 2\big\}\cap \big\{\lfloor n\alpha\rfloor: n\in \mathbb{N}\big\}\Big|=\infty.
\end{align}
It is of interest to determine whether the density $\frac{1}{k}-\frac{1}{k\alpha}$ of $A$ satisfying (\ref{eq-1-1}) is optimal for all $k\ge 2$. We will show that this is in fact the case.

\begin{theorem}\label{thm1}
Let $\alpha>1$ be an irrational number and $k\ge 2$ a positive integer. Let $f(x)$ be a polynomial with positive integer coefficients. Let $A$ be an infinite sequence of positive integers with $\overline{d}(A)>\frac{1}{k}-\frac{1}{k\alpha}$. Then
\begin{align*}
\Big|\big\{f(a_1)+\ldots+f(a_k): a_i\in A, 1\le i\le k\big\}\cap \big\{\lfloor n\alpha\rfloor: n\in \mathbb{N}\big\}\Big|=\infty.
\end{align*}
\end{theorem}

\begin{corollary}\label{hegyvari}
  For any $k\ge 2$, the density $\frac{1}{k}-\frac{1}{k\alpha}$ of $A$  satisfying $(\ref{eq-1-1})$ is optimal. 
\end{corollary}

The proof of Theorem \ref{thm1} makes use of the following two lemmas. Let $\mathbb{Z}$ be the set of integers. The first lemma is Weyl's equidistribution theorem \cite{Weyl1, Weyl2}.

\begin{lemma}[Weyl]\label{weyl}
Let $f(x)\in \mathbb{Z}[x]$ and $\rho$ be an irrational number. Let $J\subseteq(0,1)$ be a finite union of intervals. Then the asymptotic density of the set 
$$
H=\big\{n\in \mathbb{N}: \{\rho f(n)\}\in J\big\}
$$
exists and $d(H)=\mu(J)$.
\end{lemma}

Let $\mu^*$ and $\mu_*$ be the Lebesgue outer and inner measures, respectively.  We need the following result of Raikov \cite{Raikov} (see also Ruzsa \cite[4.1 Proposition]{Ruzsa-new}). 

\begin{proposition}\label{proposition}
Let $\mathbb{R}$ be the set of real numbers and $\mathbb{T}=\mathbb{R}/\mathbb{Z}$. Then for any $\mathcal{A}, \mathcal{B}\subset\mathbb{T}$ we have
$$
\mu^*(\mathcal{A}+\mathcal{B})\ge \min\big\{\mu^*(\mathcal{A})+\mu_*(\mathcal{B}), 1\big\}.
$$
\end{proposition}

Our second lemma is a consequence of Proposition \ref{proposition}.

\begin{lemma}\label{Brunn-M}
Let $k\ge 2$ be a given integer and $0<\beta<1$ a given number. Let $A\subset (0,1)$ be a set with $\mu^*(A)>\beta/k$ and $B\subset (0,1)$ be an open set with $\mu(B)\ge 1-\beta$. Then there exist $k$ elements $\alpha_1, \alpha_2,\ldots, \alpha_k$ of $A$ $($not necessarily different$)$ such that
\begin{align*}
\alpha_1+\alpha_2+\cdots+\alpha_k \pmod{1}\in B.
\end{align*}
\end{lemma}

\begin{proof}
The following argument is given by Imre Z. Ruzsa. All of the sets discussed below can be naturally regarded as subsets of $\mathbb{T}$.
Since $B$ is open, we know that $a-B$ is open for any $a\in A$, where $a-B=\big\{a-b: b\in B\big\}$. Thus, the union of open sets
$$
A-B=\bigcup_{a\in A}(a-B)
$$
is also open. It then follows from the inductive process that all the sets
$$
jA-B=\{a_1+a_2+\cdots+a_j-b: a_i\in A,\ 1\le i\le j,\ b\in B\} \quad (1\le j\le k)
$$
are open. Hence, $jA-B~(1\le j\le k)$ are measurable. 
By Proposition \ref{proposition} we have
\begin{align}\label{eq-m-1}
\mu\Big((k-1)A-B\Big)&=\mu^*\Big((k-1)A-B\Big)\nonumber\\
&\ge \min\Big\{ (k-1)\mu^*(\mathcal{A})+\mu_*(-\mathcal{B}), 1\Big\}\nonumber\\
&=\min\Big\{ (k-1)\mu^*(\mathcal{A})+\mu(\mathcal{B}), 1\Big\}.
\end{align}
We are going to show $0\in kA-B$, from which our desired result clear follows. Assume the contrary, i.e., 
$0\not\in kA-B$. Then $-a\not\in (k-1)A-B$ for any $a\in A$ which means 
\begin{align*}
-A\subset \mathbb{T}\setminus \Big((k-1)A-B\Big).
\end{align*}
Therefore, we conclude that 
\begin{align}\label{eq-m-2}
\mu^*(A)=\mu^*(-A)\le \mu\Big(\mathbb{T}\setminus \big((k-1)A-B\big)\Big)= 1-\mu\Big((k-1)A-B\Big).
\end{align}
Since $\mu^*(A)\ge \beta/k>0$, we see from (\ref{eq-m-1}) and (\ref{eq-m-2}) that
$$
(k-1)\mu^*(\mathcal{A})+\mu(\mathcal{B})<1
$$
and hence
\begin{align}\label{eq-m-3}
\mu\Big((k-1)A-B\Big)\ge (k-1)\mu^*(\mathcal{A})+\mu(\mathcal{B}).
\end{align}
By (\ref{eq-m-2}) and (\ref{eq-m-3}) we get
$$
k\mu^*(A)+\mu(\mathcal{B})\le 1,
$$
which is certainly a contradiction with
$$
k\mu^*(A)+\mu(\mathcal{B})>k\cdot\frac{\beta}{k}+1-\beta=1,
$$
proving our lemma.
\end{proof}

\begin{proof}[Proof of Theorem \ref{thm1}]
We will prove this conclusion by contradiction. Suppose that there is an infinite $A$ with $\overline{d}(A)>\frac{1}{k}-\frac{1}{k\alpha}$ such that
\begin{align*}
\Big|\big\{f(a_1)+\ldots+f(a_k): a_i\in A, 1\le i\le k\big\}\cap \big\{\lfloor n\alpha\rfloor: n\in \mathbb{N}\big\}\Big|<\infty.
\end{align*}
Deleting finitely many elements of $A$, we can assume without loss of generality that
\begin{align}\label{condition}
\big\{f(a_1)+\ldots+f(a_k): a_i\in A, 1\le i\le k\big\}\cap \big\{\lfloor n\alpha\rfloor: n\in \mathbb{N}\big\}=\emptyset.
\end{align}
Let 
$$
A_\alpha:=\left\{\frac{f(a)}{\alpha} \pmod{1}: a\in A\right\}.
$$
We claim that 
\begin{align}\label{outer}
    \mu^*(A_\alpha)\le \frac{1}{k}-\frac{1}{k\alpha}.
\end{align}
Suppose the contrary, i.e., $\mu^*(A_\alpha)> \frac{1}{k}-\frac{1}{k\alpha}$. Then by Lemma \ref{Brunn-M} with $\beta=1-\frac{1}{\alpha}$ there are $k$ elements $\alpha_1, \ldots, \alpha_k$ of $A_\alpha$ such that
\begin{align}\label{eq:Thm1-1}
    \alpha_1+\cdots+\alpha_k \pmod{1}\in \left(1-\frac{1}{\alpha}, 1\right).
\end{align}
By the definition of $A_\alpha$, there are $k$ elements $a_1, \ldots, a_k$ of $A$ such that for any $1\le i\le k$ we have
\begin{align}\label{eq:Thm1-2}
\frac{f(a_i)}{\alpha} \pmod{1}=\alpha_i.
\end{align}
By (\ref{eq:Thm1-1}) and (\ref{eq:Thm1-2}), we have
\begin{align*}
\frac{f(a_1)}{\alpha} +\cdots+ \frac{f(a_k)}{\alpha}\pmod{1}\in \left(1-\frac{1}{\alpha}, 1\right).
\end{align*}
It follows that there is a positive integer $n$ such that
\begin{align*}
  n -\frac{1}{\alpha} <\frac{f(a_1)}{\alpha} +\cdots+ \frac{f(a_k)}{\alpha}<n.
\end{align*}
Hence, we get
$$
f(a_1) +\cdots+ f(a_k)=\lfloor \alpha n\rfloor,
$$
which is a contradiction to (\ref{condition}). So, the claim (\ref{outer}) is proved.

Now, by (\ref{outer}) and the definition of Lebesgue outer measure, for any $\epsilon>0$ there exist finitely many open intervals $I_j$, where $1\le j\in J $ such  that
\begin{align}\label{eq:new-1}
    A_\alpha\subset \bigcup_{ 1\le j\le J} I_j \quad \text{and} \quad \mu\Big(\bigcup_{1\le j\le J} I_j\Big)< \frac{1}{k}-\frac{1}{k\alpha}+\varepsilon.
\end{align}
We choose 
$$
\varepsilon=\frac{\overline{d}(A)-(\frac{1}{k}-\frac{1}{k \alpha})}{2},
$$
and now consider  
$$
A^*=\left\{n\in \mathbb{N}: \frac{f(n)}{\alpha} \pmod{1}\in\bigcup_{1\le j\le J} I_j \right\}.
$$
Then $A\subset A^*$ since $A_\alpha\subset \cup_{1\le j\le J} I_j$. By Lemma \ref{weyl} we know that the density of $A^*$ exists and 
$$
d(A^*)=\mu\Big(\bigcup_{1\le j\le J} I_j\Big)< \frac{1}{k}-\frac{1}{k\alpha}+\varepsilon<\overline{d}(A),
$$
which contradicts $A\subset A^*$.
\end{proof}

\section*{Acknowledgments}
We thank Changhao Chen for his interest in this article. We are grateful to the one who brought our attention to the results in \cite{Raikov} and \cite{Ruzsa-new}, and to Imre Z. Ruzsa for his very detailed explanations on the results in \cite{Ruzsa-new}.

Y.D. is supported by National Natural Science Foundation of China  (Grant No. 12201544) and China Postdoctoral Science Foundation (Grant No. 2022M710121).
H.L. is supported by  National Natural Science Foundation of China (Grant No. 12201313).

\end{document}